\documentclass[preprint,12pt]{elsarticle}

\usepackage{amssymb}
\usepackage{amsmath}
\usepackage{amsthm}

\usepackage{tikz}
\usepackage{subcaption}
\usetikzlibrary{calc, backgrounds} 

\newtheorem{thm}{Theorem}

\newdefinition{definition}{Definition}

\newtheorem{theoremA}{Theorem}

\journal{Operations Research Letters}

\begin{document}

\begin{frontmatter}


\title{
Double Traversals in Optimal Picker Routes \\
for Warehouses with Multiple Blocks
} 


\affiliation[label1]{
    organization={
    School of Information and Physical Sciences, 
    University of Newcastle
    },
    state={NSW},
    country={Australia}
}
\affiliation[label2]{
    organization={
    Department of Industrial and Management Systems Engineering,\\
    University of South Florida
    },
    state={FL},
    country={USA}
}
\affiliation[label3]{
    organization={
    Carey Business School,
    Johns Hopkins University
    },
    state={MD},
    country={USA}
}
\affiliation[label4]{
    organization={
    International Computer Science Institute,\\
    University of California at Berkeley
    },
    state={CA},
    country={USA}
}

\author[label1]{George Dunn\corref{cor1}}
\ead{george.dunn@uon.edu.au}

\author[label2]{Hadi Charkhgard}
\author[label3,label4]{Ali Eshragh}
\author[label1]{Elizabeth Stojanovski}

\cortext[cor1]{Corresponding author}


\begin{abstract}
Order picking is a process that involves collecting items from their
respective locations within a warehouse.
There exist dynamic programming algorithms for finding the
minimal picker route by considering only a limited number
of options for possible travel within a subaisle.
Although one such action,
traversing an aisle twice,
has been shown to never be required
for a rectangular warehouse with
two cross-aisles,
this is not the case when there are more than two
cross-aisles.
In this work,
we demonstrate that double traversals within a subaisle
are not required to connect cross-aisle
travel regardless of the number of cross-aisles.
This result simplifies the structure of
feasible tours,
enabling more efficient algorithms.
\end{abstract}






\begin{keyword}
Warehousing
\sep
Picker routing
\sep
Tour subgraphs
\sep
Dynamic programming
\end{keyword}

\end{frontmatter}




\section{Introduction}
\label{introduction}



The order picking problem involves determining the
shortest route through a warehouse that visits
all item locations in a specified pick list.
\citet{ratliff1983order} presented a dynamic programming
algorithm with linear complexity
that solves this problem
in the case of a warehouse with a parallel aisle structure
and two cross-aisles \citep{hessler2022note}.
This work was then extended by
\citet{roodbergen2001routing}
to the case of three cross-aisles,
and further generalized by \citet{pansart2018exact}
to accommodate any number of cross-aisles.
All three algorithms are based on
the assumption that
a minimal picker route involves only a
limited number of possible
options to travel within an aisle.
According to \citet{ratliff1983order},
in rectangular warehouse layouts with two cross-aisles,
any minimal distance picking route does not require traversing
the same aisle more than once.
A more extensive proof is provided by \citet{revenant2024note},
along with examples of such actions being
required for a minimal picker route in
warehouses with more than two cross-aisles.
In this letter,
we demonstrate that although
double edges are necessary for some minimal routes,
they are never required for connecting
cross-aisle travel.
This insight simplifies route optimization
by eliminating redundant aisle movements.


We consider a parallel aisle warehouse with
$m \geq 1$ vertical aisles and
$n \geq 2$ horizontal cross-aisles,
as depicted in Figure \ref{fig:two_warehouse_figure}, 
when $m = 4$ and $n = 3$.
Items are stored in subaisles,
the section of the aisle between two cross-aisles,
which are assumed to be sufficiently
narrow such that the horizontal distance required
to travel from one side
to the other is negligible.
The warehouse can then be represented as a graph
$G = (V \cup P, E)$
as illustrated in Figure \ref{fig:figure_warehouse_graph},
with vertices $v_{i,j} \in V$
at the intersection of aisle $i \in [1, m]$ and
cross-aisle $j \in [1, n]$, respectively.
A set of vertices, $P = \{p_0, p_1,  ..., p_k\}$,
represents the locations to be visited,
with $p_0$ as the depot where the picker must
start and finish the tour;
while $\{p_1, ..., p_k\}$ represents the products to be collected.
Only $p_0$ can be located at a $v_{i,j}$ vertex,
while all other vertices of $P$ are located within the subaisles.
The picker routing problem involves finding the
shortest tour of $G$ that visits all points in $P$.
Note that a directed graph is not required,
as constructing a tour from an undirected graph
is straightforward \citep{ratliff1983order}.
Lastly, we assume
that $G$ has a rectangular shape, that is,
all subaisles between two cross-aisles
have the same length $d_{j,j+1}^{aisle}$ for $1 \leq j < n$ and
the horizontal distances between aisles have a constant
length $d_{i,i+1}^{cross}$ for $1 \leq i < m$.

\begin{figure}[ht]
\centering
    \begin{minipage}{0.32\textwidth}
        \centering
        \resizebox{\textwidth}{!}{\begin{tikzpicture}[shorten >=1pt,draw=black!50]

    \draw[black, thin] (0, -1) -- (0 , 6) -- (6, 6) -- (6, -0.5)
    -- (1.5, -0.5) -- (1.5, -1) -- cycle;

    \foreach \name / \x in {0,...,3}{
    
        \draw[black, thin] (1.5 * \x, 0) -- (1.5 * \x , 2.5);
        \draw[black, thin] (1.5 * \x + 0.5, 0) -- (1.5 * \x + 0.5 , 2.5);
        \draw[black, thin] (1.5 * \x + 1, 0) -- (1.5 * \x + 1 , 2.5);
        \draw[black, thin] (1.5 * \x + 1.5, 0) -- (1.5 * \x + 1.5, 2.5);

        \draw[black, thin] (1.5 * \x, 3) -- (1.5 * \x , 5.5);
        \draw[black, thin] (1.5 * \x + 0.5, 3) -- (1.5 * \x + 0.5 , 5.5);
        \draw[black, thin] (1.5 * \x + 1, 3) -- (1.5 * \x + 1 , 5.5);
        \draw[black, thin] (1.5 * \x + 1.5, 3) -- (1.5 * \x + 1.5, 5.5);
        
        \foreach \name / \y in {0,...,11}{
            \draw[black, thin] (1.5 * \x, 0.5 * \y) -- (1.5 * \x + 0.5, 0.5 * \y);
            \draw[black, thin] (1.5 * \x + 1, 0.5 * \y) -- (1.5 * \x + 1.5, 0.5 * \y);
            }
    }


    \node[align=center] at (0.75, -0.75) 
    {\small Depot};

    \node[align=center] (cross) at (3, 5.75) {\small Cross-Aisle};
    \draw[->] (cross) -- (5, 5.75);
    \draw[->] (cross) -- (1, 5.75);
    
    \node[align=center, rotate=90] (aisle) at (0.75, 2.5) {\small Aisle};
    \draw[->] (aisle) -- (0.75, 5.5);
    \draw[->] (aisle) -- (0.75, 0);

    \node[align=center, rotate=90] (subaisle1) at (5.25, 1.25) {\small Subaisle};
    \draw[->] (subaisle1) -- (5.25, 2.5);
    \draw[->] (subaisle1) -- (5.25, 0);
    \node[align=center, rotate=90] (subaisle2) at (5.25, 4.25) {\small Subaisle};
    \draw[->] (subaisle2) -- (5.25, 5.5);
    \draw[->] (subaisle2) -- (5.25, 3);


    \node[circle,fill=black,minimum size=3pt] (item1) at (0.25, 1.25) {};

    \node[circle,fill=black,minimum size=3pt] (item2) at (1.25, 3.25) {};

    \node[circle,fill=black,minimum size=3pt] (item3) at (2.75, 1.75) {};

    \node[circle,fill=black,minimum size=3pt] (item4) at (2.75, 4.25) {};

    \node[circle,fill=black,minimum size=3pt] (item5) at (3.25, 2.25) {};

    \node[circle,fill=black,minimum size=3pt] (item6) at (4.25, 5.25) {};

    \node[circle,fill=black,minimum size=3pt] (item7) at (4.75, 0.75) {};

    \node[circle,fill=black,minimum size=3pt] (item8) at (5.75, 4.75) {};

\end{tikzpicture}}
        \caption{Layout.}
        \label{fig:two_warehouse_figure}
    \end{minipage}
    \begin{minipage}{0.32\textwidth}
        \centering
        \resizebox{\textwidth}{!}{\begin{tikzpicture}[shorten >=1pt,draw=black!50]


    \draw[black, thick, double, double distance between line centers=5pt] (1.5, -0.25) -- (6, -0.25);

    \draw[black, thick, double, double distance between line centers=5pt] (1.5, 2.75) -- (6, 2.75);

    \draw[black, thick, double, double distance between line centers=5pt] (1.5, 5.75) -- (6, 5.75);

    \foreach \name / \y in {1,...,4}{
        \node[shape=circle,draw=black, minimum size=20pt, fill=white] (A-\name) at (1.5 * \y, 5.75) {};
        \node[align=center] at (A-\name) {$v_{\name, 3}$};
        
        \node[shape=circle,draw=black, minimum size=20pt, fill=white] (B-\name) at (1.5 * \y , 2.75) {};
        \node[align=center] at (B-\name) {$v_{\name, 2}$};

        \node[shape=circle,draw=black, minimum size=20pt, fill=white] (C-\name) at (1.5 * \y , -0.25) {};
        \node[align=center] at (C-\name) {$v_{\name, 1}$};

        \draw[black, thick, double, double distance between line centers=5pt] (A-\name) -- (B-\name) -- (C-\name);
        }


    \node[align=center] at (1.95, 0.2) {$p_0$};

    \node[circle, draw=black, minimum size=20pt, fill=white] (item1) at (1.5, 1.25) {};
    \node[align=center] at (item1) {$p_1$};

    \node[circle,draw=black, minimum size=20pt, fill=white] (item2) at (1.5, 3.75) {};
    \node[align=center] at (item2) {$p_2$};

    \node[circle,draw=black, minimum size=20pt, fill=white] (item3) at (3, 1.5) {};
    \node[align=center] at (item3) {$p_3$};

    \node[circle,draw=black, minimum size=20pt, fill=white] (item4) at (3, 4.25) {};
    \node[align=center] at (item4) {$p_4$};

    \node[circle,draw=black, minimum size=20pt, fill=white] (item5) at (4.5, 1.75) {};
    \node[align=center] at (item5) {$p_5$};

    \node[circle,draw=black, minimum size=20pt, fill=white] (item6) at (4.5, 4.75) {};
    \node[align=center] at (item6) {$p_6$};

    \node[circle,draw=black, minimum size=20pt, fill=white] (item7) at (6, 0.75) {};
    \node[align=center] at (item7) {$p_7$};

    \node[circle,draw=black, minimum size=20pt, fill=white] (item8) at (6, 4.5) {};
    \node[align=center] at (item8) {$p_8$};

    \node (left) at (1, 0) { };
    \node (right) at (6.5, 0) { };

\end{tikzpicture}}
        \caption{Graph $G$.}   
        \label{fig:figure_warehouse_graph}
    \end{minipage}
    \begin{minipage}{0.32\textwidth}
        \centering
        \resizebox{\textwidth}{!}
        {\begin{tikzpicture}[shorten >=1pt,draw=black!50]





    \foreach \name / \y in {1,...,4}{
        \node[shape=circle,draw=black, minimum size=20pt, fill=white] (A-\name) at (1.5 * \y, 5.75) {};
        \node[align=center] at (A-\name) {$v_{\name, 3}$};
        
        \node[shape=circle,draw=black, minimum size=20pt, fill=white] (B-\name) at (1.5 * \y , 2.75) {};
        \node[align=center] at (B-\name) {$v_{\name, 2}$};

        \node[shape=circle,draw=black, minimum size=20pt, fill=white] (C-\name) at (1.5 * \y , -0.25) {};
        \node[align=center] at (C-\name) {$v_{\name, 1}$};

        }


    \node[align=center] at (1.95, 0.2) {$p_0$};

    \node[circle, draw=black, minimum size=20pt, fill=white] (item1) at (1.5, 1.25) {};
    \node[align=center] at (item1) {$p_1$};

    \node[circle,draw=black, minimum size=20pt, fill=white] (item2) at (1.5, 3.75) {};
    \node[align=center] at (item2) {$p_2$};

    \node[circle,draw=black, minimum size=20pt, fill=white] (item3) at (3, 1.5) {};
    \node[align=center] at (item3) {$p_3$};

    \node[circle,draw=black, minimum size=20pt, fill=white] (item4) at (3, 4.25) {};
    \node[align=center] at (item4) {$p_4$};

    \node[circle,draw=black, minimum size=20pt, fill=white] (item5) at (4.5, 1.75) {};
    \node[align=center] at (item5) {$p_5$};

    \node[circle,draw=black, minimum size=20pt, fill=white] (item6) at (4.5, 4.75) {};
    \node[align=center] at (item6) {$p_6$};

    \node[circle,draw=black, minimum size=20pt, fill=white] (item7) at (6, 0.75) {};
    \node[align=center] at (item7) {$p_7$};

    \node[circle,draw=black, minimum size=20pt, fill=white] (item8) at (6, 4.5) {};
    \node[align=center] at (item8) {$p_8$};

    \draw[black, thick, double, double distance between line centers=5pt]
    (item2) -- (B-1) -- (item1) -- (C-1) -- (C-2);

    \draw[black, thick, double, double distance between line centers=5pt]
    (item6) -- (A-3);

    \draw[black, thick, double, double distance between line centers=5pt]
    (item5) -- (C-3);

    \draw[black, thick]
    (C-2) -- (item3) -- (B-2) -- (item4) -- (A-2) --
    (A-3) -- (A-4) --
    (item8) -- (B-4) -- (item7) -- (C-4) --
    (C-3) -- (C-2);

    \node (left) at (1, 0) { };
    \node (right) at (6.5, 0) { };

\end{tikzpicture}}
        \caption{Subgraph $T$.}   
        \label{fig:figure_warehouse_subgraph}
    \end{minipage}
\end{figure}


For a minimal picker route,
there are only six possible
vertical edge configurations within each subaisle
and three horizontal edge configurations between
each cross-aisle,
as shown in Figure \ref{fig:vertical_action}
and Figure \ref{fig:horizontal_action},
respectively
\citep{ratliff1983order, roodbergen2001routing, pansart2018exact}.
Although vertical configuration $(v)$ has been shown
not to be required for a rectangular warehouse with
two cross-aisles,
there are cases where it is required for warehouses
with more that two aisles.
In this letter,
we show that a minimal picker route will never
require double vertical edges between
vertices with incident horizontal edges.
In Section \ref{prerequisites}, 
some preliminary results and definitions are introduced,
before proving the main theorem in
Section \ref{proof_v3}.

\begin{figure}[ht]
\centering
    \begin{minipage}{0.4\textwidth}
        \centering
        \resizebox{\textwidth}{!}
        {\begin{tikzpicture}[shorten >=1pt,->,draw=black!50, node distance=\layersep]

\tikzset{minimum size=34pt}

    \node[shape=circle,draw=black] (a1) at (0, 5) { };
    \node at (a1) {$v_{i,j+1}$};
    \node[shape=circle,draw=black] (d1) at (0, 3.5) { };
    \node[shape=circle,draw=black] (c1) at (0, 1.5) { };
    \node[shape=circle,draw=black] (b1) at (0, 0) {$v_{i,j}$}; 
    \draw[-, thick, draw=black] (a1) -- (d1) -- (c1) -- (b1);

    \node (i) at (0, -1) {(i)};

    \node[shape=circle,draw=black] (a2) at (1.5, 5) { };
    \node at (a2) {$v_{i, j+1}$};
    \node[shape=circle,draw=black] (d2) at (1.5, 3.5) { };
    \node[shape=circle,draw=black] (c2) at (1.5, 1.5) { };
    \node[shape=circle,draw=black] (b2) at (1.5, 0) {$v_{i, j}$}; 
    \draw[-, thick, draw=black, double, double distance between line centers=8pt] (a2) -- (d2) -- (c2);

    \node (ii) at (1.5, -1) {(ii)};

    \node[shape=circle,draw=black] (a3) at (3, 5) { };
    \node at (a3) {$v_{i, j+1}$};
    \node[shape=circle,draw=black] (d3) at (3, 3.5) { };
    \node[shape=circle,draw=black] (c3) at (3, 1.5) { };
    \node[shape=circle,draw=black] (b3) at (3, 0) {$v_{i, j}$}; 
    \draw[-, thick, draw=black, double, double distance between line centers=8pt] (b3) -- (c3) -- (d3);

    \node (iii) at (3, -1) {(iii)};

    \node[shape=circle,draw=black] (a4) at (4.5, 5) { };
    \node at (a4) {$v_{i, j+1}$};
    \node[shape=circle,draw=black] (d4) at (4.5, 3.5) { };
    \node[shape=circle,draw=black] (c4) at (4.5, 1.5) { };
    \node[shape=circle,draw=black] (b4) at (4.5, 0) {$v_{i, j}$}; 
    \draw[-, thick, draw=black, double, double distance between line centers=8pt] (a4) -- (d4);
    \draw[-, thick, draw=black, double, double distance between line centers=8pt] (b4) -- (c4);

    \node (iv) at (4.5, -1) {(iv)};

    \node[shape=circle,draw=black] (a5) at (6, 5) { };
    \node at (a5) {$v_{i, j+1}$};
    \node[shape=circle,draw=black] (d5) at (6, 3.5) { };
    \node[shape=circle,draw=black] (c5) at (6, 1.5) { };
    \node[shape=circle,draw=black] (b5) at (6, 0) {$v_{i, j}$}; 
    \draw[-, thick, draw=black, double, double distance between line centers=8pt] (a5) -- (d5) -- (c5) -- (b5);

    \node (v) at (6, -1) {(v)};

    \node[shape=circle,draw=black] (a6) at (7.5, 5) { };
    \node at (a6) {$v_{i, j+1}$};
    \node[shape=circle,draw=black] (d6) at (7.5, 3.5) { };
    \node[shape=circle,draw=black] (c6) at (7.5, 1.5) { };
    \node[shape=circle,draw=black] (b6) at (7.5, 0) {$v_{i, j}$}; 

    \node (vi) at (7.5, -1) {(vi)};

    \node (left) at (-0.5, 0) { };
    \node (right) at (8, 0) { };

\end{tikzpicture}}
        \caption{Vertical configurations.} 
        \label{fig:vertical_action}
    \end{minipage}
    \begin{minipage}{0.4\textwidth}
        \centering
        \resizebox{\textwidth}{!}
        {\begin{tikzpicture} 

\tikzset{minimum size=34pt}

    \node[shape=circle,draw=black] (al1) at (0, 5) { };
    \node at (al1) {\small $v_{i, j}$}; 
    \node[shape=circle,draw=black] (ar1) at (2, 5) { };
    \node at (ar1) {\small $v_{i+1, j}$}; 
    \node (1) at (1, 4.0) {(i)};

    \node[shape=circle,draw=black] (al2) at (0, 2.5) { };
    \node at (al2) {\small $v_{i, j}$}; 
    \node[shape=circle,draw=black] (ar2) at (2, 2.5) { };
    \node at (ar2) {\small $v_{i+1, j}$}; 
    \draw[-, thick] (al2) -- (ar2);
    \node (2) at (1, 1.5) {(ii)};

    \node[shape=circle,draw=black] (al3) at (0, 0) { };
    \node at (al3) {\small $v_{i, j}$}; 
    \node[shape=circle,draw=black] (ar3) at (2, 0) { };
    \node at (ar3) {\small $v_{i+1, j}$}; 
    \draw[-, thick, draw=black, double, double distance between line centers=8pt] (al3) -- (ar3);
    \node (3) at (1, -1) {(iii)};

    \node (4) at (-3, 0) { };
    \node (5) at (5.5, 0) { };

\end{tikzpicture}}
        \caption{Horizontal configurations.}
        \label{fig:horizontal_action}
    \end{minipage}
\end{figure}



\section{Prerequisites}
\label{prerequisites}


A subgraph $T \subseteq G$ that contains all points $P$
will be called a tour subgraph if it contains
a tour that uses every edge of $T$ exactly once,
as shown in Figure \ref{fig:figure_warehouse_subgraph}.
The following theorem and corollaries
defining the characteristics of a tour subgraph,
are based on a well-known theorem
on Eulerian graphs
\citep{ratliff1983order, christofides1975graph}.

\begin{theoremA}[Ratliff and Rosenthal, 1983]
\label{thm:subtour}
A subgraph $T \subseteq G$
is a tour subgraph if and only if:
\begin{enumerate}[(i)]
    \item All vertices of $P$ belong to the vertices of $T$;
    \item $T$ is connected; and 
    \item Every vertex in $T$ has an even degree.
\end{enumerate}
\end{theoremA}


We now define configurations that span one or
more adjacent subaisles.
Let $(v_{i,j}, v_{i,k})^1$ be a single
vertical edge between vertices such that
all subaisles
between contain a single edge and
no vertices between have horizontal
incident edges.
With edge multiplicity $m(v_1, v_2)$
as the number of edges between two vertices,
this gives:
\begin{enumerate}[(i)]
\item $m(v_{i,l}, v_{i,l+1}) = 1$,
for all $j \leq l < k$, and
\item
$m(v_{i-1,l}, v_{i,l}) + m(v_{i,l}, v_{i+1,l}) = 0$,
for all $j < l < k$.
\end{enumerate}

The definition extends naturally to
a double edge by considering
two edges in each subaisle.
A double edge,
$(v_{i,j}, v_{i,k})^2$,
is called connecting if
it connects horizontal edges,
that is:
\begin{enumerate}[(i)]
    \item $m(v_{i-1,j}, v_{i,j}) + m(v_{i,j}, v_{i+1,j}) > 0$,
    and
    \item $m(v_{i-1,k}, v_{i,k}) + m(v_{i,k}, v_{i+1,k}) > 0$,
\end{enumerate}


In the next section,
we will prove
that a connecting double edge
is not required in a minimal
tour subgraph.

\begin{thm}
\label{thm:main} 
There exists a minimum length tour subgraph
$T \subseteq G$ that does not contain a
connecting double edge.
\end{thm}



\section{Proof}
\label{proof_v3}


In this section,
we prove the main result
by providing a state representation for
connecting double edges
and proving that all possible states can be
reduced.
For simplicity,
we refer to two distinct vertices
in aisle $i$ as $a_i$ and $b_i$.
The state of the connecting double edge $(a_i, b_i)^2$ 
can be represented as the number of
edges incident to the left of each vertex:
\begin{equation*}
    s = (m(a_{i-1}, a_i), m(b_{i-1}, b_i)) \in \mathcal{S}.
\end{equation*}
Without loss of generality,
we assume $m(b_{i-1}, b_{i}) > 0$,
as the warehouse can be flipped
vertically around aisle $i$
to ensure that this is always the case.
Due to horizontal symmetry,
the proof for the state $(s_a, s_b)$
also applies to $(s_b, s_a)$.
Therefore, it is sufficient to consider
only the following
five distinct combinations:
\begin{equation*}
    \mathcal{S} = \{ (0, 1), (0, 2), (1, 1), (1, 2), (2, 2) \}.
\end{equation*}

\begin{figure}[ht]
\begin{center}
\centering
    \begin{subfigure}[b]{0.17\textwidth}
        \centering
        \resizebox{\linewidth}{!}{
            \begin{tikzpicture}[shorten >=1pt,draw=black!50]

    \pgfmathsetmacro{\x}{2}
    \pgfmathsetmacro{\y}{4}
    \pgfmathsetmacro{\z}{1.25}

    \pgfmathsetmacro{\a}{0.25}

    \pgfmathsetmacro{\b}{0.5}

    \pgfmathsetmacro{\nodesize}{25}

    
    \node[circle, draw=black, minimum size=\nodesize pt, fill=white]
    (a_j-1) at (0, \y + \z) { };
    \node[align=center] at (a_j-1) {$a_{i-1}$};

    \node[circle, draw=black, minimum size=\nodesize pt, fill=white]
    (a_j) at (\x, \y + \z) { };
    \node[align=center] at (a_j) {$a_i$};

    \node[circle, draw=black, minimum size=\nodesize pt, fill=white]
    (b_j) at (\x, \z) { };
    \node[align=center] at (b_j) {$b_i$};

    \node[circle, draw=black, minimum size=\nodesize pt, fill=white]
    (b_j-1) at (0, \z) { };
    \node[align=center] at (b_j-1) {$b_{i-1}$};

    \node[circle, draw=black, minimum size=\nodesize pt, fill=white]
    (p_1) at (0, 0.5*\y + \z) { };

    \node[circle, draw=black, minimum size=\nodesize pt, fill=white]
    (p_2) at (\x, 0.5*\y + \z) { };

    \draw[-, draw=black, thick, double, double distance between line centers=8pt]
    (a_j) -- (p_2) -- (b_j);

    \draw[-, draw=black, thick]
    (b_j-1) -- (b_j);

    \begin{pgfonlayer}{background}
        \fill[gray!20]
        ($(a_j-1) + (\b, 0)$) --
        ($(b_j-1) + (\b, 0)$) --
        ($(b_j-1) + (-\b, 0)$) --
        ($(a_j-1) + (-\b, 0)$) --
        cycle;
    \end{pgfonlayer}


    \draw[->, draw=black, thick] (0.5*\x, \a) -- (0.5*\x, -\a);


    \node[circle, draw=black, minimum size=\nodesize pt, fill=white]
    (a2_j-1) at (0, -\z) { };
    \node[align=center] at (a2_j-1) {$a_{i-1}$};

    \node[circle, draw=black, minimum size=\nodesize pt, fill=white]
    (a2_j) at (\x, -\z) { };
    \node[align=center] at (a2_j) {$a_i$};

    \node[circle, draw=black, minimum size=\nodesize pt, fill=white]
    (b2_j) at (\x, -\y -\z) { };
    \node[align=center] at (b2_j) {$b_i$};

    \node[circle, draw=black, minimum size=\nodesize pt, fill=white]
    (b2_j-1) at (0, -\y -\z) { };
    \node[align=center] at (b2_j-1) {$b_{i-1}$};

    \node[circle, draw=black, minimum size=\nodesize pt, fill=white]
    (p2_1) at (0, -0.5*\y - \z) { };

    \node[circle, draw=black, minimum size=\nodesize pt, fill=white]
    (p2_2) at (\x, -0.5*\y - \z) { };

    \draw[-, draw=black, thick]
    (a2_j) -- (p2_2) -- (b2_j);

    \draw[-, draw=red, thick]
    (a2_j-1) -- (p2_1) -- (b2_j-1);

    \draw[-, draw=black, thick]
    (a2_j-1) -- (a2_j);

    \begin{pgfonlayer}{background}
        \fill[gray!20]
        ($(a2_j-1) + (\b, 0)$) --
        ($(b2_j-1) + (\b, 0)$) --
        ($(b2_j-1) + (-\b, 0)$) --
        ($(a2_j-1) + (-\b, 0)$) --
        cycle;
    \end{pgfonlayer}

\end{tikzpicture}
        }
        \caption{$(0, 1)$}
        \label{fig:transformation_01}
    \end{subfigure}
    \begin{subfigure}[b]{0.17\textwidth}
        \centering
        \resizebox{\linewidth}{!}{
            \begin{tikzpicture}[shorten >=1pt,draw=black!50]

    \pgfmathsetmacro{\x}{2}
    \pgfmathsetmacro{\y}{4}
    \pgfmathsetmacro{\z}{1.25}

    \pgfmathsetmacro{\a}{0.25}

    \pgfmathsetmacro{\b}{0.5}

    \pgfmathsetmacro{\nodesize}{25}

    
    \node[circle, draw=black, minimum size=\nodesize pt, fill=white]
    (a_j-1) at (0, \y + \z) { };
    \node[align=center] at (a_j-1) {$a_{i-1}$};

    \node[circle, draw=black, minimum size=\nodesize pt, fill=white]
    (a_j) at (\x, \y + \z) { };
    \node[align=center] at (a_j) {$a_i$};

    \node[circle, draw=black, minimum size=\nodesize pt, fill=white]
    (b_j) at (\x, \z) { };
    \node[align=center] at (b_j) {$b_i$};

    \node[circle, draw=black, minimum size=\nodesize pt, fill=white]
    (b_j-1) at (0, \z) { };
    \node[align=center] at (b_j-1) {$b_{i-1}$};

    \node[circle, draw=black, minimum size=\nodesize pt, fill=white]
    (p_1) at (0, 0.5*\y + \z) { };

    \node[circle, draw=black, minimum size=\nodesize pt, fill=white]
    (p_2) at (\x, 0.5*\y + \z) { };

    \draw[-, draw=black, thick, double, double distance between line centers=8pt]
    (a_j) -- (p_2) -- (b_j);

    \draw[-, draw=black, thick, double, double distance between line centers=8pt]
    (b_j-1) -- (b_j);

    \begin{pgfonlayer}{background}
        \fill[gray!20]
        ($(a_j-1) + (\b, 0)$) --
        ($(b_j-1) + (\b, 0)$) --
        ($(b_j-1) + (-\b, 0)$) --
        ($(a_j-1) + (-\b, 0)$) --
        cycle;
    \end{pgfonlayer}


    \draw[->, draw=black, thick] (0.5*\x, \a) -- (0.5*\x, -\a);


    \node[circle, draw=black, minimum size=\nodesize pt, fill=white]
    (a2_j-1) at (0, -\z) { };
    \node[align=center] at (a2_j-1) {$a_{i-1}$};

    \node[circle, draw=black, minimum size=\nodesize pt, fill=white]
    (a2_j) at (\x, -\z) { };
    \node[align=center] at (a2_j) {$a_i$};

    \node[circle, draw=black, minimum size=\nodesize pt, fill=white]
    (b2_j) at (\x, -\y -\z) { };
    \node[align=center] at (b2_j) {$b_i$};

    \node[circle, draw=black, minimum size=\nodesize pt, fill=white]
    (b2_j-1) at (0, -\y -\z) { };
    \node[align=center] at (b2_j-1) {$b_{i-1}$};

    \node[circle, draw=black, minimum size=\nodesize pt, fill=white]
    (p2_1) at (0, -0.5*\y - \z) { };

    \node[circle, draw=black, minimum size=\nodesize pt, fill=white]
    (p2_2) at (\x, -0.5*\y - \z) { };

    \draw[-, draw=black, thick]
    (a2_j) -- (p2_2) -- (b2_j);

    \draw[-, draw=red, thick]
    (a2_j-1) -- (p2_1) -- (b2_j-1);

    \draw[-, draw=black, thick]
    (a2_j-1) -- (a2_j);

    \draw[-, draw=black, thick]
    (b2_j-1) -- (b2_j);

    \begin{pgfonlayer}{background}
        \fill[gray!20]
        ($(a2_j-1) + (\b, 0)$) --
        ($(b2_j-1) + (\b, 0)$) --
        ($(b2_j-1) + (-\b, 0)$) --
        ($(a2_j-1) + (-\b, 0)$) --
        cycle;
    \end{pgfonlayer}

\end{tikzpicture}
        }
        \caption{$(0, 2)$}
        \label{fig:transformation_02}
    \end{subfigure}
    \begin{subfigure}[b]{0.17\textwidth}
        \centering
        \resizebox{\linewidth}{!}{
            \begin{tikzpicture}[shorten >=1pt,draw=black!50]

    \pgfmathsetmacro{\x}{2}
    \pgfmathsetmacro{\y}{4}
    \pgfmathsetmacro{\z}{1.25}

    \pgfmathsetmacro{\a}{0.25}

    \pgfmathsetmacro{\b}{0.5}

    \pgfmathsetmacro{\nodesize}{25}

    
    \node[circle, draw=black, minimum size=\nodesize pt, fill=white]
    (a_j-1) at (0, \y + \z) { };
    \node[align=center] at (a_j-1) {$a_{i-1}$};

    \node[circle, draw=black, minimum size=\nodesize pt, fill=white]
    (a_j) at (\x, \y + \z) { };
    \node[align=center] at (a_j) {$a_i$};

    \node[circle, draw=black, minimum size=\nodesize pt, fill=white]
    (b_j) at (\x, \z) { };
    \node[align=center] at (b_j) {$b_i$};

    \node[circle, draw=black, minimum size=\nodesize pt, fill=white]
    (b_j-1) at (0, \z) { };
    \node[align=center] at (b_j-1) {$b_{i-1}$};

    \node[circle, draw=black, minimum size=\nodesize pt, fill=white]
    (p_1) at (0, 0.5*\y + \z) { };

    \node[circle, draw=black, minimum size=\nodesize pt, fill=white]
    (p_2) at (\x, 0.5*\y + \z) { };

    \draw[-, draw=black, thick, double, double distance between line centers=8pt]
    (a_j) -- (p_2) -- (b_j);

    \draw[-, draw=black, thick]
    (b_j-1) -- (b_j);

    \draw[-, draw=black, thick]
    (a_j-1) -- (a_j);

    \begin{pgfonlayer}{background}
        \fill[gray!20]
        ($(a_j-1) + (\b, 0)$) --
        ($(b_j-1) + (\b, 0)$) --
        ($(b_j-1) + (-\b, 0)$) --
        ($(a_j-1) + (-\b, 0)$) --
        cycle;
    \end{pgfonlayer}


    \draw[->, draw=black, thick] (0.5*\x, \a) -- (0.5*\x, -\a);


    \node[circle, draw=black, minimum size=\nodesize pt, fill=white]
    (a2_j-1) at (0, -\z) { };
    \node[align=center] at (a2_j-1) {$a_{i-1}$};

    \node[circle, draw=black, minimum size=\nodesize pt, fill=white]
    (a2_j) at (\x, -\z) { };
    \node[align=center] at (a2_j) {$a_i$};

    \node[circle, draw=black, minimum size=\nodesize pt, fill=white]
    (b2_j) at (\x, -\y -\z) { };
    \node[align=center] at (b2_j) {$b_i$};

    \node[circle, draw=black, minimum size=\nodesize pt, fill=white]
    (b2_j-1) at (0, -\y -\z) { };
    \node[align=center] at (b2_j-1) {$b_{i-1}$};

    \node[circle, draw=black, minimum size=\nodesize pt, fill=white]
    (p2_1) at (0, -0.5*\y - \z) { };

    \node[circle, draw=black, minimum size=\nodesize pt, fill=white]
    (p2_2) at (\x, -0.5*\y - \z) { };

    \draw[-, draw=black, thick]
    (a2_j) -- (p2_2) -- (b2_j);

    \draw[-, draw=red, thick]
    (a2_j-1) -- (p2_1) -- (b2_j-1);

    \draw[-, draw=black, thick, double, double distance between line centers=8pt]
    (a2_j-1) -- (a2_j);

    \begin{pgfonlayer}{background}
        \fill[gray!20]
        ($(a2_j-1) + (\b, 0)$) --
        ($(b2_j-1) + (\b, 0)$) --
        ($(b2_j-1) + (-\b, 0)$) --
        ($(a2_j-1) + (-\b, 0)$) --
        cycle;
    \end{pgfonlayer}

\end{tikzpicture}
        }
        \caption{$(1, 1)$}
        \label{fig:transformation_11}
    \end{subfigure}
    \begin{subfigure}[b]{0.17\textwidth}
        \centering
        \resizebox{\linewidth}{!}{
            \begin{tikzpicture}[shorten >=1pt,draw=black!50]

    \pgfmathsetmacro{\x}{2}
    \pgfmathsetmacro{\y}{4}
    \pgfmathsetmacro{\z}{1.25}

    \pgfmathsetmacro{\a}{0.25}

    \pgfmathsetmacro{\b}{0.5}

    \pgfmathsetmacro{\nodesize}{25}

    
    \node[circle, draw=black, minimum size=\nodesize pt, fill=white]
    (a_j-1) at (0, \y + \z) { };
    \node[align=center] at (a_j-1) {$a_{i-1}$};

    \node[circle, draw=black, minimum size=\nodesize pt, fill=white]
    (a_j) at (\x, \y + \z) { };
    \node[align=center] at (a_j) {$a_i$};

    \node[circle, draw=black, minimum size=\nodesize pt, fill=white]
    (b_j) at (\x, \z) { };
    \node[align=center] at (b_j) {$b_i$};

    \node[circle, draw=black, minimum size=\nodesize pt, fill=white]
    (b_j-1) at (0, \z) { };
    \node[align=center] at (b_j-1) {$b_{i-1}$};

    \node[circle, draw=black, minimum size=\nodesize pt, fill=white]
    (p_1) at (0, 0.5*\y + \z) { };

    \node[circle, draw=black, minimum size=\nodesize pt, fill=white]
    (p_2) at (\x, 0.5*\y + \z) { };

    \draw[-, draw=black, thick, double, double distance between line centers=8pt]
    (a_j) -- (p_2) -- (b_j);

    \draw[-, draw=black, thick, double, double distance between line centers=8pt]
    (b_j-1) -- (b_j);

    \draw[-, draw=black, thick]
    (a_j-1) -- (a_j);

    \begin{pgfonlayer}{background}
        \fill[gray!20]
        ($(a_j-1) + (\b, 0)$) --
        ($(b_j-1) + (\b, 0)$) --
        ($(b_j-1) + (-\b, 0)$) --
        ($(a_j-1) + (-\b, 0)$) --
        cycle;
    \end{pgfonlayer}


    \draw[->, draw=black, thick] (0.5*\x, \a) -- (0.5*\x, -\a);


    \node[circle, draw=black, minimum size=\nodesize pt, fill=white]
    (a2_j-1) at (0, -\z) { };
    \node[align=center] at (a2_j-1) {$a_{i-1}$};

    \node[circle, draw=black, minimum size=\nodesize pt, fill=white]
    (a2_j) at (\x, -\z) { };
    \node[align=center] at (a2_j) {$a_i$};

    \node[circle, draw=black, minimum size=\nodesize pt, fill=white]
    (b2_j) at (\x, -\y -\z) { };
    \node[align=center] at (b2_j) {$b_i$};

    \node[circle, draw=black, minimum size=\nodesize pt, fill=white]
    (b2_j-1) at (0, -\y -\z) { };
    \node[align=center] at (b2_j-1) {$b_{i-1}$};

    \node[circle, draw=black, minimum size=\nodesize pt, fill=white]
    (p2_1) at (0, -0.5*\y - \z) { };

    \node[circle, draw=black, minimum size=\nodesize pt, fill=white]
    (p2_2) at (\x, -0.5*\y - \z) { };

    \draw[-, draw=black, thick]
    (a2_j) -- (p2_2) -- (b2_j);

    \draw[-, draw=red, thick]
    (a2_j-1) -- (p2_1) -- (b2_j-1);

    \draw[-, draw=black, thick, double, double distance between line centers=8pt]
    (a2_j-1) -- (a2_j);

    \draw[-, draw=black, thick]
    (b2_j-1) -- (b2_j);

    \begin{pgfonlayer}{background}
        \fill[gray!20]
        ($(a2_j-1) + (\b, 0)$) --
        ($(b2_j-1) + (\b, 0)$) --
        ($(b2_j-1) + (-\b, 0)$) --
        ($(a2_j-1) + (-\b, 0)$) --
        cycle;
    \end{pgfonlayer}

\end{tikzpicture}
        }
        \caption{$(1, 2)$}
        \label{fig:transformation_12}
    \end{subfigure}
    \begin{subfigure}[b]{0.17\textwidth}
        \centering
        \resizebox{\linewidth}{!}{
            \begin{tikzpicture}[shorten >=1pt,draw=black!50]

    \pgfmathsetmacro{\x}{2}
    \pgfmathsetmacro{\y}{4}
    \pgfmathsetmacro{\z}{1.25}

    \pgfmathsetmacro{\a}{0.25}

    \pgfmathsetmacro{\b}{0.5}

    \pgfmathsetmacro{\nodesize}{25}

    
    \node[circle, draw=black, minimum size=\nodesize pt, fill=white]
    (a_j-1) at (0, \y + \z) { };
    \node[align=center] at (a_j-1) {$a_{i-1}$};

    \node[circle, draw=black, minimum size=\nodesize pt, fill=white]
    (a_j) at (\x, \y + \z) { };
    \node[align=center] at (a_j) {$a_i$};

    \node[circle, draw=black, minimum size=\nodesize pt, fill=white]
    (b_j) at (\x, \z) { };
    \node[align=center] at (b_j) {$b_i$};

    \node[circle, draw=black, minimum size=\nodesize pt, fill=white]
    (b_j-1) at (0, \z) { };
    \node[align=center] at (b_j-1) {$b_{i-1}$};

    \node[circle, draw=black, minimum size=\nodesize pt, fill=white]
    (p_1) at (0, 0.5*\y + \z) { };

    \node[circle, draw=black, minimum size=\nodesize pt, fill=white]
    (p_2) at (\x, 0.5*\y + \z) { };

    \draw[-, draw=black, thick, double, double distance between line centers=8pt]
    (a_j) -- (p_2) -- (b_j);

    \draw[-, draw=black, thick, double, double distance between line centers=8pt]
    (b_j-1) -- (b_j);

    \draw[-, draw=black, thick, double distance between line centers=8pt]
    (a_j-1) -- (a_j);

    \begin{pgfonlayer}{background}
        \fill[gray!20]
        ($(a_j-1) + (\b, 0)$) --
        ($(b_j-1) + (\b, 0)$) --
        ($(b_j-1) + (-\b, 0)$) --
        ($(a_j-1) + (-\b, 0)$) --
        cycle;
    \end{pgfonlayer}


    \draw[->, draw=black, thick] (0.5*\x, \a) -- (0.5*\x, -\a);


    \node[circle, draw=black, minimum size=\nodesize pt, fill=white]
    (a2_j-1) at (0, -\z) { };
    \node[align=center] at (a2_j-1) {$a_{i-1}$};

    \node[circle, draw=black, minimum size=\nodesize pt, fill=white]
    (a2_j) at (\x, -\z) { };
    \node[align=center] at (a2_j) {$a_i$};

    \node[circle, draw=black, minimum size=\nodesize pt, fill=white]
    (b2_j) at (\x, -\y -\z) { };
    \node[align=center] at (b2_j) {$b_i$};

    \node[circle, draw=black, minimum size=\nodesize pt, fill=white]
    (b2_j-1) at (0, -\y -\z) { };
    \node[align=center] at (b2_j-1) {$b_{i-1}$};

    \node[circle, draw=black, minimum size=\nodesize pt, fill=white]
    (p2_1) at (0, -0.5*\y - \z) { };

    \node[circle, draw=black, minimum size=\nodesize pt, fill=white]
    (p2_2) at (\x, -0.5*\y - \z) { };

    \draw[-, draw=black, thick]
    (a2_j) -- (p2_2) -- (b2_j);

    \draw[-, draw=red, thick]
    (a2_j-1) -- (p2_1) -- (b2_j-1);

    \draw[-, draw=black, thick, double, double distance between line centers=8pt]
    (a2_j-1) -- (a2_j);

    \draw[-, draw=black, thick]
    (a2_j-1) -- (a2_j);

    \draw[-, draw=black, thick]
    (b2_j-1) -- (b2_j);

    \begin{pgfonlayer}{background}
        \fill[gray!20]
        ($(a2_j-1) + (\b, 0)$) --
        ($(b2_j-1) + (\b, 0)$) --
        ($(b2_j-1) + (-\b, 0)$) --
        ($(a2_j-1) + (-\b, 0)$) --
        cycle;
    \end{pgfonlayer}

\end{tikzpicture}
        }
        \caption{$(2, 2)$}
        \label{fig:transformation_22}
    \end{subfigure}
\caption{Tour subgraph $T$ (top) and transformation $T'$ (bottom)
for each state.} 
\label{fig:transformations}
\end{center}
\end{figure}


We now define a transformation
$T \rightarrow T' = (V, E_{T'}) \subseteq G$,
which can be applied to a tour subgraph
containing a connecting double edge,
where:
\begin{equation*}
E_{T'} =
(E_T \! \setminus \! \{(a_i, b_i), (b_{i-1}, b_i)\})
\cup \{(a_{i-1}, b_{i-1}), (a_{i-1}, a_i)\}.
\end{equation*}


Figure \ref{fig:transformations}
illustrates the altered parts of the subgraph for each state
where the shaded regions represent the unknown configurations
in the previous aisle.
As $T$ is a valid tour subgraph,
all the conditions of Theorem \ref{thm:subtour} are met.
The transformed subgraph $T'$ remains a
valid tour subgraph because:
\begin{enumerate}[(i)]
    \item Each altered subaisle contains a continuous edge,
    ensuring that all items within it are still visited;
    \item There is a guaranteed path through
    $(b_{i-1}, a_{i-1}, a_i, b_i)$,
    ensuring that all vertices remain connected; and 
    \item For each altered vertex,
    the transformation either,
    removes two edges,
    adds two,
    or leaves the number
    of incident edges unchanged,
    thereby preserving even vertex degrees.
\end{enumerate}


Since $T'$ is a valid tour subgraph
of equal length,
it remains only to show that
the total number of subaisles involved in
connecting double edges is reduced.
If this can be proven for
all states $s \in \mathcal{S}$,
then all such configurations
can be eliminated through iterative application.


\begin{proof}[Proof of Theorem \ref{thm:main}]

We begin by observing two special cases of
connecting double edges.

\paragraph{Preliminary: Case 0.1 (redundant double edge)}
Let $\lambda(a_i, b_i)$ be the edge connectivity
between $a_i, b_i \in V$.
If $\lambda(a_i, b_i) > 2$,
removing a pair of edges between
adjacent vertices will result in
a subgraph with reduced length.
The degree parity of the vertices
and the connectivity of the graph
remain unchanged,
therefore this is still a valid
tour subgraph.
This is shown in Figure \ref{fig:redundant}
where removing the red edge pair
will eliminate the connecting double edge.
We will refer to this case as a redundant double edge.

\begin{figure}[ht]
\centering
    \begin{minipage}{\textwidth}
        \centering
        \resizebox{0.22\textwidth}{!}
        {
        \begin{tikzpicture}
        [shorten >=1pt,->,draw=black!50, node distance=\layersep]

        \tikzset{minimum size=20pt}

            \pgfmathsetmacro{\x}{1.25}
        
            \node[shape=circle,draw=black] (a1) at (0, 3*\x) { };
            \node at (a1) {$a_{i}$};
            \node[shape=circle,draw=black] (d1) at (0, 2*\x) { };
            \node[shape=circle,draw=black] (c1) at (0, \x) { };
            \node[shape=circle,draw=black] (b1) at (0, 0) { }; 
            \node at (b1) {$b_{i}$};
            \draw[-, thick, draw=black, double, double distance between line centers=8pt]
            (a1) -- (d1);
            \draw[-, thick, draw=black, double, double distance between line centers=8pt]
            (c1) -- (b1);
            \draw[-, thick, red, double, double distance between line centers=8pt]
            (d1) -- (c1);
            \draw[-, dotted, draw=black, thick]
            (a1) -- (-2*\x, 3*\x) -- (-2*\x, 0) -- (b1);
        \end{tikzpicture}
        }
        \caption{Redundant double edge.} 
        \label{fig:redundant}
    \end{minipage}
\end{figure}

\paragraph{Preliminary: Case 0.2 (no single edge in previous aisle)}
The transformation $T \rightarrow T'$
will only result in
$(a_{i-1}, b_{i-1})^2 \in T'$
if $(a_{i-1}, b_{i-1})^1 \in T$.
Therefore,
if any subaisles between these two vertices
contain a configuration other than
a single edge,
the transformation immediately reduces the
number of double edges for any state.

\medskip

\noindent
Three states can be directly
shown to be reducible as they
must belong to one of these
preliminary cases.

\paragraph{Case 1: State $s \in \{ (1, 1), (1, 2), (2, 2) \}$}
\begin{enumerate}[(i)]
    \item If $(a_{i-1}, b_{i-1})^1 \in T$,
    path $(a_i, a_{i-1}, b_{i-1}, b_i)$ exists,
    making a 
    $(a_{i}, b_{i})^2 \in T$
    double edge redundant (Case 0.1).
    \item If $(a_{i-1}, b_{i-1})^1 \notin T$,
    the transformation reduces the number
    of double edges (Case 0.2).
\end{enumerate}

\noindent
For the remaining two states,
we will assume the connecting double edges
are not redundant,
and $(a_{i-1}, b_{i-1})^1 \in T$
resulting in $(a_{i-1}, b_{i-1})^2 \in T'$.

\paragraph{Case 2: State $s = (0, 2)$}
A connecting double
$(a_{i-1}, b_{i-1})^2 \in T'$
is redundant as
path $(a_{i-1}, a_i, b_i, b_{i-1})$ exists.

\paragraph{Case 3: State $s = (0, 1)$}
Firstly, as $m(b_{i-1}, b_i) = 0$ in $T'$,
it is possible for a double edge
to extend from $a_{i-1}$
to below $b_{i-1}$.
We will refer to the lowest vertex
in the double edge as $c_{i-1}$,
which could still be $b_{i-1}$.
For a double edge
$(a_{i-1}, c_{i-1})^2 \in T'$,
there are three possibilities:
\begin{enumerate}[(i)]
    \item There are no horizontal edges
    incident to $c_{i-1}$,
    therefore the double edge is not connecting
    and we have reduced the number of
    connecting double edges.
    \item There are incident edges to the
    right of $c_{i-1}$.
    Flipping the warehouse by vertical
    symmetry along aisle $i - 1$
    shows this edge belongs to state
    $(1, 1)$ or $(1, 2)$
    and therefore,
    can be reduced (Case 1).
    \item There are only incident edges to the
    left of $c_{i-1}$.
    The connecting double edge has therefore
    not been reduced, but
    shifted one aisle to the left.
\end{enumerate}

\noindent
Given we have now addressed all possible cases,
it has been shown that a connecting double
edge can either be reduced or shifted to the left.
A connecting double edge in the first (or last)
aisle must have incident edges in one direction,
therefore belonging to state
$(1, 1)$, $(1, 2)$ or $(2, 2)$,
which are reducible.
This shows that a series of transformations
must terminate in, or before,
the first aisle.

\end{proof}


\section{Conclusion}

We have shown that minimal picker tours
in rectangular warehouses with any number of
cross-aisles do not require vertical double edges
to connect components containing horizontal edges.
This is a significant result,
as existing algorithms
for this problem consider
single and double edges as valid configurations
influencing tour graph connectivity
\citep{ratliff1983order, roodbergen2001routing, pansart2018exact}.
By simplifying the structure of feasible routes,
more efficient algorithms can be developed.
Furthermore, the results of this letter also
apply directly to both horizontal and vertical
edges in the algorithm for the related
rectilinear traveling salesman problem presented
by \citet{cambazard2018fixed}.


\newpage

\section*{Acknowledgments}

George was supported by an Australian Government
Research Training Program (RTP) Scholarship.


%


\bibliographystyle{elsarticle-num-names} 
\bibliography{bibliography}

\end{document}